\documentclass[a4paper,12pt,egregdoesnotlikesansseriftitles,abstract,twoside=semi,headinclude=true,leqno]{scrartcl}
\usepackage{amsmath,amsthm,amssymb,xcolor,scrlayer-scrpage,titling,enumitem}
\usepackage[T1]{fontenc}
\usepackage{lmodern,exscale}
\usepackage[colorlinks, linkcolor=purple, citecolor=teal]
{hyperref}

\newtheoremstyle{myplain}
  {}
  {}
  {\itshape}
  {}
  {\bfseries}
  {.}
  {5pt plus 1pt minus 1pt}
  {}

\newtheorem{theorem}{Theorem}[section]
\newtheorem{proposition}[theorem]{Proposition}
\newtheorem{lemma}[theorem]{Lemma}
\newtheorem{corollary}[theorem]{Corollary}
\newtheoremstyle{defi}
  {}
  {}
  {}
  {}
  {\scshape}
  {.}
  {5pt plus 1pt minus 1pt}
  {}

\theoremstyle{definition}
\newtheorem{definition}[theorem]{Definition}
\theoremstyle{remark}
\newtheorem{remark}[theorem]{Remark}

\numberwithin{equation}{section}

\title{\textbf{
Strong instability of standing waves for nonlinear Schr\"odinger equations with attractive inverse power potential
}}
\author{Noriyoshi Fukaya and Masahito Ohta}
\date{}

\addtokomafont{pagehead}{\upshape}
\cehead{N. Fukaya and M. Ohta}
\cohead[]{Strong instability of standing waves}
\lehead*[]{\pagemark}
\rohead[]{\pagemark}
\lefoot[]{}
\rofoot[]{}

\begin{document}
\maketitle
\thispagestyle{plain.scrheadings}
\begin{abstract}
We study the strong instability of standing waves $e^{i\omega t}\phi_\omega(x)$ for nonlinear Schr\"{o}dinger equations with an $L^2$-supercritical nonlinearity and an attractive inverse power potential,
where $\omega\in\mathbb{R}$ is a frequency, 
and $\phi_\omega\in H^1(\mathbb{R}^N)$ is a ground state of the corresponding stationary equation.
Recently, 
for nonlinear Schr\"odinger equations with a harmonic potential,
Ohta~(2018) proved that if $\partial_\lambda^2S_\omega(\phi_\omega^\lambda)|_{\lambda=1}\le0$,
then the standing wave is strongly unstable,
where $S_\omega$ is the action, and
$\phi_\omega^\lambda(x)\mathrel{\mathop:}=\lambda^{N/2}\phi_\omega(\lambda x)$
is the scaling, which does not change the $L^2$-norm.
In this paper, we prove the strong instability under the same assumption as the above-mentioned in inverse power potential case.
Our proof is applicable to nonlinear Schr\"odinger equations with other potentials such as an attractive Dirac delta potential.
\end{abstract}
\footnotetext[0]{2010 \textit{Mathematics Subject Classification}. 
35Q55,
35B35}
\section{Introduction}

In this paper,
we consider the nonlinear Schr\"odinger equation with an attractive inverse power potential
\[\tag{NLS}\label{nls}
i\partial_tu
=-\Delta u
-\frac{\gamma}{|x|^{\alpha}}u
-|u|^{p-1}u,\quad
(t,x)\in\mathbb{R}\times\mathbb{R}^N,
\]
where 
\begin{gather}\label{asmp}
N\in\mathbb{N},\quad
\gamma>0,\quad
0<\alpha<\min\{2,N\},\quad
1+\frac{4}{N}<p<1+\frac4{N-2},
\end{gather}
and $u\colon\mathbb{R}\times\mathbb{R}^{N}\to\mathbb{C}$ is an unknown function of $(t,x)\in\mathbb{R}\times\mathbb{R}^N$.
Here, $1+4/(N-2)$ stands for $\infty$ if $N=1$ or $2$.


Let us consider the Cauchy problem for \eqref{nls}.
Since the potential $V(x)\mathrel{\mathop:}=-\gamma|x|^{-\alpha}$ belongs to $(L^{r}+L^\infty)(\mathbb{R}^N)$
for some $r>\min\{1,N/2\}$ under the assumption~\eqref{asmp},
the multiplication operator $v\mapsto V(x)v$ is continuous from $H^1(\mathbb{R}^N)$ to $(L^{\rho'}+L^2)(\mathbb{R}^N)$ for some 
$\rho\in[2,2N/(N-2))$,
and thus, the potential energy $\int_{\mathbb{R}^N} V(x)|v(x)|^2\,dx$ is well-defined on $H^1(\mathbb{R}^N)$.
Therefore, the local well-posedness of \eqref{nls} in the energy space $H^1(\mathbb{R}^N)$ follows from the standard theory, 
e.g.\ \cite[Theorems~3.3.5, 3.3.9, Proposition~4.2.3]{Cazenave}.
More precisely,
for each $u_0\in H^1(\mathbb{R}^N)$,
there exist a maximal interval $I_{\max}=[0,T^+)\subset\mathbb{R}$
with $T^+=T^+(u_0)\in(0,\infty]$ and a unique solution 
$u\in C(I_{\max},H^1(\mathbb{R}^N))$ of \eqref{nls} with $u(0)=u_0$ such that if $T^+<\infty$,
then 
$\lim_{t\nearrow T^+}\|u(t)\|_{H^1}=\infty$.
Here, if $T^+<\infty$,
we say that the solution $u(t)$ \textit{blows up in finite time}.
Moreover,
\eqref{nls} satisfies the two conservation laws 
\begin{gather*}
E(u(t))=E(u_0),\quad
\|u(t)\|_{L^2}
=\|u_0\|_{L^2}
\end{gather*}
for all $t\in I_{\max}$,
where
\[
E(v)
\mathrel{\mathop:}=\frac12\|\nabla v\|_{L^2}^2
-\frac{\gamma}2\int_{\mathbb{R}^N}\frac{|v(x)|^2}{|x|^\alpha}\,dx
-\frac{1}{p+1}\|v\|_{L^{p+1}}^{p+1}
\]
is the energy.

By a \textit{standing wave},
we mean a solution of \eqref{nls} with the form $e^{i\omega t}\phi(x)$,
where $\omega\in\mathbb{R}$ is a frequency,
and $\phi\in H^1(\mathbb{R}^N)$ is a nontrivial solution of the stationary equation
\begin{equation}\label{sp}
-\Delta\phi
+\omega\phi
-\frac{\gamma}{|x|^{\alpha}}\phi
-|\phi|^{p-1}\phi
=0,\quad
x\in\mathbb{R}^N.
\end{equation}
Eq.\ \eqref{sp} can be written as $S_\omega'(\phi)=0$,
where 
\[
S_\omega(v)
\mathrel{\mathop:}=E(v)+\frac{\omega}{2}\|v\|_{L^2}^2
\]
is the action.
The following existence and variational characterization of ground states by using the Nehari functional
\begin{align*}
K_\omega(v)
&\mathrel{\mathop:}=\partial_\lambda S_\omega(\lambda v)|_{\lambda=1}
=\langle S_\omega'(v),v\rangle \\
&=\|\nabla v\|_{L^2}^2
+\omega\|v\|_{L^2}^2
-\gamma\int_{\mathbb{R}^N}\frac{|v(x)|^2}{|x|^\alpha}\,dx
-\|v\|_{L^{p+1}}^{p+1}
\end{align*}
are known (see \cite[Remarks~1.2 and 1.3]{FO-s}),
where a \textit{ground state} is a nontrivial solution of \eqref{sp} with the least action.

\begin{proposition} \label{prop:GS}
Assume \eqref{asmp} and 
\begin{equation}\label{omega}
\omega>
\omega_0
\mathrel{\mathop:}=-\inf\left\{\,\|\nabla v\|_{L^2}^2-\gamma\int_{\mathbb{R}^N}\frac{|v(x)|^2}{|x|^\alpha}\,dx
\mathrel{}\middle|\mathrel{}
v\in H^1(\mathbb{R}^N),~\|v\|_{L^2}=1\,\right\}.
\end{equation}
Then the set of ground states
\[
\mathcal{G}_\omega
\mathrel{\mathop:}=\{\,\phi\in\mathcal{F}_\omega\mid S_\omega(\phi)\le S_\omega(v)~\text{for all $v\in\mathcal{F}_\omega$}\,\}
\]
is not empty,
where
\[
\mathcal{F}_\omega
\mathrel{\mathop:}=\{\,\phi\in H^1(\mathbb{R}^N)\setminus\{0\}\mid S_\omega'(\phi)=0\,\}
\]
is the set of all nontrivial solutions of \eqref{sp}.
Moreover, if $\phi\in\mathcal{G}_\omega$, then
\begin{equation}
S_\omega(\phi)
=\inf\{\,S_\omega(v)\mid v\in H^1(\mathbb{R}^N)\setminus\{0\},~
K_\omega(v)=0\,\}.
\end{equation}
\end{proposition}
For the sake of completeness,
we give a proof of Proposition~\ref{prop:GS} in Section~\ref{sec:GS} by using the argument in \cite[Section~3]{FOO08}.

In the present paper,
we study the strong instability of the standing wave solution $e^{i\omega t}\phi_\omega$ of \eqref{nls},
where $\omega>\omega_0$ and $\phi_\omega\in\mathcal{G}_\omega$.
We recall the definitions of stability and instability of standing waves.

\begin{definition}
Let $e^{i\omega t}\phi$ be a standing wave solution of \eqref{nls}.
\begin{itemize}
\item
We say that $e^{i\omega t}\phi$ is \textit{stable} if for each $\varepsilon>0$, 
there exists $\delta>0$ such that if $u_0\in H^1(\mathbb{R}^N)$ satisfies 
$\|u_0-\phi\|_{H^1}<\delta$,
then the solution $u(t)$ of \eqref{nls} with $u(0)=u_0$ exists globally in time,
and satisfies
\[
\sup_{t\ge0}\inf_{\theta\in\mathbb{R}}\|u(t)-e^{i\theta}\phi\|_{H^1}<\varepsilon.
\]
\item
We say that $e^{i\omega t}\phi$ is \textit{unstable} if $e^{i\omega t}\phi$ is not stable.
\item
We say that $e^{i\omega t}\phi$ is \textit{strongly unstable} if for each $\varepsilon>0$, there exists $u_0\in H^1(\mathbb{R}^N)$ such that $\|u_0-\phi\|_{H^1}<\varepsilon$
and the solution $u(t)$ of \eqref{nls} with $u(0)=u_0$ blows up in finite time.
\end{itemize}
\end{definition}

Here, we state some known results related to our works.
The stability and instability of standing waves with a ground state profile for nonlinear Schr\"odinger equations have been studied by many researchers.
For \eqref{nls} in the nonpotential case $\gamma=0$,
Berestycki and Cazenave \cite{BC81} proved the strong instability for any $\omega>0$ when $1+4/N\le p<1+4/(N-2)$
(for the case $p=1+4/N$, see also \cite{Weinstein83}).
Cazenave and Lions~\cite{CL82} proved the stability for any $\omega>0$ if $1<p<1+4/N$.
For abstract Hamiltonian systems including nonlinear Schr\"odinger equations,
Grillakis, Shatah, and Strauss \cite{GSS87, GSS90} gave sufficient conditions for the stability and instability,
that is,
if $\partial_\omega\|\phi_\omega\|_{L^2}^2>0$, the standing wave is stable,
and if $\partial_\omega\|\phi_\omega\|_{L^2}^2<0$,
the standing wave is unstable (see also \cite{Shatah83, SS85, Weinstein86}).
For the nonlinear Schr\"odinger equation with a general potential
\begin{equation}\label{nlsp}
i\partial_tu
=-\Delta u
+\tilde V(x)u
-|u|^{p-1}u,\quad
(t,x)\in\mathbb{R}\times\mathbb{R}^N,
\end{equation}
Rose and Weinstein~\cite{RW88} proved the stability for $\omega>\tilde\omega_0$ sufficiently closed to $\tilde\omega_0$ even when $1+4/N\le p<1+4/(N-2)$ by using the criteria of Grillakis, Shatah, and Strauss~\cite{GSS87},
where $-\tilde\omega_0$ is the smallest eigenvalue of the Schr\"odinger operator $-\Delta+\tilde V$.
In \cite{FO-s},
Ohta and Fukuizumi improved the stability results of Rose and Weinstein,
and in \cite{FO-i},
they proved the instability for sufficiently large $\omega$ when $1+4/N<p<1+4/(N-2)$ by using the sufficient condition of Ohta~\cite{Ohta95},
that is,
if $\partial_\lambda^2\tilde S_\omega(\phi_\omega^\lambda)|_{\lambda=1}<0$, the standing wave is unstable,
where 
$\tilde S_\omega$ is the action corresponding to \eqref{nlsp}, 
and $v^\lambda(x)\mathrel{\mathop:}=\lambda^{N/2}v(\lambda x)$ is the scaling, 
which does not change the $L^2$-norm 
(see also \cite{FOO08,GHP04} in the Dirac delta potential case and \cite{Fukuizumi01} in the harmonic potential case).
For the nonlinear Schr\"odinger equation with an attractive Dirac delta potential
\begin{equation}\label{nlsd}
i\partial_tu
=-\partial_x^2u
-\tilde\gamma\delta(x)u
-|u|^{p-1}u,\quad
(t,x)\in\mathbb{R}\times\mathbb{R},
\end{equation}
Ohta and Yamaguchi~\cite{OY16} proved the strong instability of the standing wave with positive energy $\tilde E(\phi_\omega)>0$ when $\tilde\gamma>0$ and $p>5$,
and as a corollary,
they proved the strong instability for sufficiently large $\omega$
(see also \cite{OY15} for related works).
Recently,
for the nonlinear Schr\"odinger equation with a harmonic potential
\begin{equation}\label{nlsh}
i\partial_tu
=-\Delta u
+|x|^2u
-|u|^{p-1}u,\quad
(t,x)\in\mathbb{R}\times\mathbb{R}^N,
\end{equation}
Ohta~\cite{Ohtapre} proved the strong instability under the same assumption $\partial_\lambda^2\tilde S_\omega(\phi_\omega^\lambda)|_{\lambda=1}\le0$ as in \cite{Ohta95} when $1+4/N<p<1+4/(N-2)$.

In view of the graph of $\lambda\mapsto \tilde S_\omega(\phi_\omega^\lambda)$,
we see that $\tilde E(\phi_\omega)>0$ implies $\partial_\lambda^2 \tilde S_\omega(\phi_\omega^\lambda)|_{\lambda=1}<0$.
Therefore, the question naturally arises whether the standing wave is strongly unstable or not in the case $\tilde E(\phi_\omega)\le0$ and $\partial_\lambda^2\tilde S_\omega(\phi_\omega^\lambda)|_{\lambda=1}\le0$
for \eqref{nlsd}.
However, the proof for \eqref{nlsh} in \cite{Ohtapre} is not applicable to \eqref{nlsd}.

In this paper,
we consider the strong instability of standing waves under the same assumption as in \cite{Ohtapre}.
In order to treat more general potentials with suitable properties related to the scaling $\lambda\mapsto v^\lambda$,
we study the nonlinear Schr\"odinger equation \eqref{nls} with an inverse power potential.
Now, we state our main result.

\begin{theorem}\label{mainthm}
Assume
\eqref{asmp},
$\omega>\omega_0$,
and that $\phi_\omega\in\mathcal{G}_\omega$ satisfies 
$\partial_\lambda^2S_\omega(\phi_\omega^\lambda)|_{\lambda=1}\le0$,
where $\phi_\omega^\lambda(x)=\lambda^{N/2}\phi_\omega(\lambda x)$.
Then the standing wave solution $e^{i\omega t}\phi_\omega$ of \eqref{nls} is strongly unstable.
\end{theorem}

It is proven in \cite[Section~2]{FO-i} that the assumption $\partial_\lambda^2S_\omega(\phi_\omega^\lambda)|_{\lambda=1}\le0$ is satisfied for sufficiently large $\omega$.
Therefore, we have the following corollary.

\begin{corollary}
Assume
\eqref{asmp}.
Then there exists $\omega_1>\omega_0$ such that if $\omega\ge\omega_1$ and $\phi_\omega\in\mathcal{G}_\omega$,
the standing wave solution $e^{i\omega t}\phi_\omega$ of \eqref{nls} is strongly unstable.
\end{corollary}

\begin{remark}
Theorem~\ref{mainthm} can be extended to more general settings.
The important feature used in the proof of Theorem~\ref{mainthm} is that the energy satisfies 
\begin{gather} \label{eq:gn1}
E(v)=\frac12\|\nabla v\|_{L^2}^2
-\frac12G(v)
-\frac{1}{p+1}\|v\|_{L^{p+1}}^{p+1},\\ \label{eq:gn3}
G(v)\ge0,\quad
G(\lambda v)
=\lambda^2G(v),\quad
G(v^\lambda)
=\lambda^\alpha G(v),\quad
\|v^\lambda\|_{L^{p+1}}^{p+1}
=\lambda^\beta\|v\|_{L^{p+1}}^{p+1}
\end{gather}
with $\beta>2>\alpha>0$.
Since the energy of \eqref{nlsd} satisfies \eqref{eq:gn1} and \eqref{eq:gn3} with 
$G(v)=\gamma|v(0)|^2$,
$\alpha=1$, and
$\beta=(p-1)/2$,
the proof is applicable to \eqref{nlsd} for $p>5$.
This gives an improvement of the result of Ohta and Yamaguchi~\cite{OY16}.
\end{remark}


The proof of blowup for nonlinear Schr\"odinger equations relies on the virial identity
\begin{equation}\label{virial}
\frac{d^2}{dt^2}\|xu(t)\|_{L^2}^2
=8Q(u(t)),
\end{equation}
where $Q$ is the functional on $H^1(\mathbb{R}^N)$ defined by
\begin{align*}
Q(v)
=\|\nabla v\|_{L^2}^2
-\frac{\gamma\alpha}{2}\int_{\mathbb{R}^N}\frac{|v(x)|^2}{|x|^\alpha}\,dx
-\frac{N(p-1)}{2(p+1)}\|v\|_{L^{p+1}}^{p+1}.
\end{align*}
Note that
\begin{align*}
S_\omega(v^\lambda)
&=\frac{\lambda^2}{2}\|\nabla v\|_{L^2}^2
+\frac{\omega}{2}\|v\|_{L^2}^2
-\frac{\gamma\lambda^\alpha}{2}\int_{\mathbb{R}^N}\frac{|v(x)|^2}{|x|^\alpha}\,dx
-\frac{\lambda^{N(p-1)/2}}{p+1}\|v\|_{L^{p+1}}^{p+1}, \\
Q(v)
&=\partial_{\lambda}S_\omega(v^\lambda)|_{\lambda=1}.
\end{align*}
Since
$x\cdot\nabla V(x)=\gamma\alpha|x|^{-\alpha}\in(L^{q}+L^\infty)(\mathbb{R}^N)$
for some $q>\min\{1,N/2\}$ under the assumption~\eqref{asmp},
from the standard theory \cite[Proposition~6.5.1]{Cazenave},
we obtain the local well-posedness of the Cauchy problem for \eqref{nls} in the weighted space
\[
\Sigma
\mathrel{\mathop:}=\{\,v\in H^1(\mathbb{R}^N)\mid \|xv\|_{L^2}<\infty\,\},
\]
and the virial identity~\eqref{virial}
holds for all $t\in I_{\max}$.

To prove Theorem~\ref{mainthm},
we introduce the set
\begin{align*}
\mathcal{B}_\omega
=\left\{\,v\in H^1(\mathbb{R}^N)
\mathrel{}\middle|\mathrel{}
\begin{gathered}
S_\omega(v)<S_\omega(\phi_\omega),~\|v\|_{L^2}\le\|\phi_\omega\|_{L^2}, \\
\|v\|_{L^{p+1}}>\|\phi_\omega\|_{L^{p+1}},~
Q(v)<0
\end{gathered}\,\right\}.
\end{align*}
Then we have the following blowup result.

\begin{theorem}\label{blowup}
Assume
\eqref{asmp},
$\omega>\omega_0$,
and that $\phi_\omega\in\mathcal{G}_\omega$ satisfies 
$\partial_\lambda^2S_\omega(\phi_\omega^\lambda)|_{\lambda=1}\le0$.
If $u_0\in\mathcal{B}_\omega\cap\Sigma$,
then the solution $u(t)$ of \eqref{nls} with $u(0)=u_0$ blows up in finite time.
\end{theorem}

Theorem~\ref{mainthm} follows from Theorem~\ref{blowup} and the fact that the ground state $\phi_\omega$ belongs to the closure of $\mathcal{B}_\omega\cap\Sigma$ in $H^1$-topology.

The key to the proof of Theorem~\ref{blowup} is Lemma~\ref{key-lemma} below.
The same assertion of Lemma~\ref{key-lemma} is proven in \cite[Lemma~4]{Ohtapre} for \eqref{nlsh}.
In \cite[Lemma~4]{Ohtapre},
the proof is divided into two cases
$\|x\phi_\omega\|_{L^2}^2\le\|xv\|_{L^2}^2$ and
$\|xv\|_{L^2}^2\le\|x\phi_\omega\|_{L^2}^2$.
Although the first case is easy to treat,
the second case is more complicated.
In the second case, 
the inequality $\|xv\|_{L^2}^2\le\|x\phi_\omega\|_{L^2}^2$ is used
to obtain upper bounds for the potential energy.
However, in our case, 
this argument does not work well because the sign of the potential is different from that of \eqref{nlsh}.
In our proof here,
to obtain upper bounds for the potential energy,
we use the inequality coming out of the variational characterization of the ground state (see~Lemma~\ref{vari-ineq}~\eqref{vari-ineq1} below).

We remark that
in \cite{Ohtapre,OY16},
they consider 
\[
\left\{\,v\in H^1(\mathbb{R}^N)
\mathrel{}\middle|\mathrel{}
\begin{gathered}
\tilde E(v)<\tilde E(\phi_\omega),~\|v\|_{L^2}=\|\phi_\omega\|_{L^2}, \\
\|v\|_{L^{p+1}}>\|\phi_\omega\|_{L^{p+1}},~\tilde Q(v)<0
\end{gathered}\,\right\}\cap\Sigma
\]
as the set of initial data of blowup solutions.
On the other hand,
in our definition of $\mathcal{B}_\omega$, we use the action $S_\omega$ instead of the energy $E$ in order to treat more general initial data.

We finally remark that the assumption $\partial_\lambda^2S_\omega(\phi_\omega^\lambda)|_{\lambda=1}\le0$ is not a necessary condition for the instability because it is known for \eqref{nlsd} that there exist unstable standing waves satisfying $\partial_\lambda^2\tilde S_\omega(\phi_\omega^\lambda)|_{\lambda=1}>0$ (see~\cite[Section~4]{OY16}).
It is an open problem whether the standing wave is strongly unstable or not in this case.

This paper is organized as follows.
In Section~\ref{sec:GS},
we give a proof of Proposition~\ref{prop:GS}
and prove a useful lemma (Lemma~\ref{vari-ineq} below).
In Section~\ref{sec:blowup},
we prove Theorem~\ref{blowup}.
In Section~\ref{sec:SI},
we prove Theorem~\ref{mainthm}.

\section{Existence and Variational Characterization of ground states}\label{sec:GS}
 
The aim of this section is to prove Proposition~\ref{prop:GS} and Lemma~\ref{vari-ineq} below.
Here, we assume \eqref{asmp} and $\omega>\omega_0$,
where $\omega_0$ is defined in \eqref{omega}.
Hereafter,
we denote
\begin{equation}\label{eq:Gv}
G(v)
=\gamma\int_{\mathbb{R}^N}\frac{|v(x)|^2}{|x|^\alpha}\,dx.
\end{equation}

We define
\begin{align*}
d(\omega) 
&=\inf\{\,S_\omega(v)\mid v\in H^1(\mathbb{R}^N)\setminus\{0\},~K_\omega(v)=0\,\}, \\
\mathcal{M}_\omega 
&=\{\,v\in H^1(\mathbb{R}^N)\setminus\{0\}\mid K_\omega(v)=0,~S_\omega(v)=d(\omega)\,\}.
\end{align*}
Note that since $-\omega_0$ is the smallest eigenvalue of the Schr\"odinger operator
$-\Delta-\gamma|x|^{-\alpha}$,
under the assumption $\omega>\omega_0$,
we have the equivalence of norms
\begin{equation} \label{equi-norm}
\sqrt{L_\omega(v)}
\simeq\|v\|_{H^1},
\end{equation}
where
\[
L_\omega(v)
=\|\nabla v\|_{L^2}^2
+\omega\|v\|_{L^2}^2
-G(v).
\]

First,
we show that ground states of \eqref{sp} are characterized as the minimizers for $S_\omega$ under the constraint $K_\omega=0$.

\begin{lemma} \label{lem:2.1}
$\mathcal{M}_\omega\subset\mathcal{G}_\omega$.
\end{lemma}

\begin{proof}
Let $\phi\in\mathcal{M}_\omega$.
Then by $L_\omega(\phi)-\|\phi\|_{L^{p+1}}^{p+1}
=K_\omega(\phi)=0$,
we have
\begin{equation} \label{eq:1}
\langle K_\omega'(\phi),\phi\rangle
=2L_\omega(\phi)-(p+1)\|\phi\|_{L^{p+1}}^{p+1}
=-(p-1)\|\phi\|_{L^{p+1}}^{p+1}
<0.
\end{equation}
Therefore,
there exists a Lagrange multiplier $\eta\in\mathbb{R}$ such that
$S_\omega'(\phi)=\eta K_\omega'(\phi)$.
Moreover, since
\[
\eta\langle K_\omega'(\phi),\phi\rangle
=\langle S_\omega'(\phi),\phi\rangle
=K_\omega(\phi)
=0,
\]
it follows from \eqref{eq:1} that $\eta=0$,
which implies $S_\omega'(\phi)=0$.

Furthermore, if $v\in H^1(\mathbb{R}^N)$ satisfies $v\ne0$ and $S_\omega'(v)=0$, 
then by
$K_\omega(v)=\langle S_\omega'(v),v\rangle=0$
and the definition of $\mathcal{M}_\omega$,
we have
$S_\omega(\phi)\le S_\omega(v)$.
Thus, we obtain $\phi\in\mathcal{G}_\omega$.
This completes the proof.
\end{proof}

\begin{lemma}
If $\mathcal{M}_\omega$ is not empty,
then $\mathcal{G}_\omega\subset\mathcal{M}_\omega$.
\end{lemma}

\begin{proof}
Let $\phi\in\mathcal{G}_\omega$.
Since $\mathcal{M}_\omega$ is not empty,
we take $\psi\in\mathcal{M}_\omega$.
Then by Lemma~\ref{lem:2.1},
we have $\psi\in\mathcal{G}_\omega$.
Therefore, if $v\in H^1(\mathbb{R}^N)$ satisfies $v\ne0$ and $K_\omega(v)=0$,
then
$S_\omega(\phi)
=S_\omega(\psi)
\le S_\omega(v)
$.
This implies $\phi\in\mathcal{M}_\omega$.
This completes the proof.
\end{proof}

Next, we show that $\mathcal M_\omega$ is not empty.
By using 
\begin{align} \label{skn}
S_{\omega}(v)
&=\frac{1}{2}K_{\omega}(v)
+\frac{p-1}{2(p+1)}\|v\|_{L^{p+1}}^{p+1} \\ \notag
&=\frac{1}{p+1}K_{\omega}(v)
+\frac{p-1}{2(p+1)}L_\omega(v),
\end{align}
we rewrite
\begin{align} \label{eq:don}
d(\omega)
&=\inf\left\{\,\frac{p-1}{2(p+1)}\|v\|_{L^{p+1}}^{p+1}
\mathrel{}\middle|\mathrel{}
v\in H^1(\mathbb{R}^N)\setminus\{0\},~K_{\omega}(v)=0\,\right\} \\
&=\inf\left\{\,\frac{p-1}{2(p+1)}L_\omega(v)
\mathrel{}\middle|\mathrel{}
v\in H^1(\mathbb{R}^N)\setminus\{0\},~K_{\omega}(v)=0\,\right\}. \label{eq:dol}
\end{align}

\begin{lemma}\label{lem:k1}
If $K_\omega(v)<0$,
then 
\[
\frac{p-1}{2(p+1)}\|v\|_{L^{p+1}}^{p+1}
>d(\omega),\quad
\frac{p-1}{2(p+1)}L_\omega(v)
>d(\omega).
\]
In particular,
\begin{align}\notag
d(\omega)
&=\inf\left\{\,\frac{p-1}{2(p+1)}\|v\|_{L^{p+1}}^{p+1}
\mathrel{}\middle|\mathrel{}
v\in H^1(\mathbb{R}^N)\setminus\{0\},~K_\omega(v)\le0\,\right\} \\ \label{eq:dkl}
&=\inf\left\{\,\frac{p-1}{2(p+1)}L_\omega(v)
\mathrel{}\middle|\mathrel{}
v\in H^1(\mathbb{R}^N)\setminus\{0\},~K_\omega(v)\le0\,\right\}.
\end{align}
\end{lemma}

\begin{proof}
Let
\[
\lambda_1=\left(\frac{L_\omega(v)}{\|v\|_{L^{p+1}}^{p+1}}\right)^{1/(p-1)},
\]
where note that $L_\omega(v)>0$ by \eqref{equi-norm}.
Then since
$
K_\omega(\lambda v)
=\lambda^2L_\omega(v)
-\lambda^{p+1}\|v\|_{L^{p+1}}^{p+1}
$
and $K_\omega(v)<0$,
we have $K_\omega(\lambda_1v)=0$ and $0<\lambda_1<1$.
Therefore, by \eqref{eq:don},
\[
d(\omega)
\le\frac{p-1}{2(p+1)}\|\lambda_1 v\|_{L^{p+1}}^{p+1}
=\lambda_1^{p+1}\frac{p-1}{2(p+1)}\|v\|_{L^{p+1}}^{p+1}
<\frac{p-1}{2(p+1)}\|v\|_{L^{p+1}}^{p+1}.
\]
Similarly, by using \eqref{eq:dol}, we obtain
$d(\omega)
<\frac{p-1}{2(p+1)}L_\omega(v)$.
This completes the proof.
\end{proof}

It is well known that in the nonpotential case $\gamma=0$,
the set of all minimizers 
\begin{equation*} 
\mathcal{M}_{\omega}^0\mathrel{\mathop:}=\{\,v\in H^1(\mathbb{R}^N)\setminus\{0\}\mid K_\omega^0(v)=0,~S_\omega^0(v)=d^0(\omega)\,\}
\end{equation*}
is not empty (see e.g.\ \cite{LeCoz08, Lions84}), where 
\begin{align*}
S_\omega^0(v)
&=\frac12\|\nabla v\|_{L^2}^2
+\frac\omega2\|v\|_{L^2}^2
-\frac{1}{p+1}\|v\|_{L^{p+1}}^{p+1}, \\
K_\omega^0(v)
&=\|\nabla v\|_{L^2}^2
+\omega\|v\|_{L^2}^2
-\|v\|_{L^{p+1}}^{p+1}, \\
d^0(\omega) 
&=\inf\{\,S_\omega^0(v)\mid v\in H^1(\mathbb{R}^N)\setminus\{0\},~K_\omega^0(v)=0\,\} \\
&=\inf\left\{\,\frac{p-1}{2(p+1)}\|v\|_{L^{p+1}}^{p+1}
\mathrel{}\middle|\mathrel{}
v\in H^1(\mathbb{R}^N)\setminus\{0\},~K_\omega^0(v)=0\,\right\}.
\end{align*}

\begin{lemma}\label{lem:do}
$d^0(\omega)>d(\omega)>0$.
\end{lemma}

\begin{proof}
First, we show $d^0(\omega)>d(\omega)$.
Since $\mathcal{M}_\omega^0$ is not empty, 
we take $\psi\in\mathcal{M}_\omega^0$.
Since 
\[
K_\omega(\psi)
=K_\omega^0(\psi)
-G(\psi)
=-G(\psi)
<0,
\]
by Lemma~\ref{lem:k1},
we have
\[
d(\omega)
<\frac{p-1}{2(p+1)}\|\psi\|_{L^{p+1}}^{p+1}
=d^0(\omega).
\]

Next, we show that $d(\omega)>0$.
Let $v\in H^1(\mathbb{R}^N)$ satisfy $v\ne0$ and $K_{\omega}(v)=0$.
By the Sobolev embedding,
\eqref{equi-norm},
and $L_{\omega}(v)=\|v\|_{L^{p+1}}^{p+1}$,
we have
\[
\|v\|_{L^{p+1}}^2
\le C_1\|v\|_{H^1}^2
\le C_2L_{\omega}(v)
=C_2\|v\|_{L^{p+1}}^{p+1}.
\]
for some $C_1,C_2>0$.
Since $v\ne0$,
we have $\|v\|_{L^{p+1}}\ge C_2^{-1/(p-1)}$.
Taking the infimum over $v$,
we obtain $d(\omega)>0$.
This completes the proof.
\end{proof}

\begin{lemma}
Let $(v_n)_n\subset H^1(\mathbb{R}^N)$ be a minimizing sequence for $d(\omega)$, 
that is, 
\[
v_n\ne0,\quad
K_\omega(v_n)=0,\quad
S_\omega(v_n)\to d(\omega).
\]
Then there exist a subsequence $(v_{n_k})_k$ of $(v_{n})_n$ and $v_0\in H^1(\mathbb{R}^N)$ such that
$v_{n_k}\to v_0$ in $H^1(\mathbb{R}^N)$,
$K_\omega(v_0)=0$,
and $S_\omega(v_0)=d(\omega)$.
In particular, 
$\mathcal{M}_\omega$ is not empty.
\end{lemma}

\begin{proof}
First, by $K_\omega(v_n)=0$,
$S_\omega(v_n)\to d(\omega)$,
and \eqref{skn},
we have
\begin{equation}\label{eq:lnd}
\frac{p-1}{2(p+1)}L_\omega(v_n)
=\frac{p-1}{2(p+1)}\|v_n\|_{L^{p+1}}^{p+1}
\to d(\omega).
\end{equation}
Therefore, it follows from \eqref{equi-norm} that $(v_n)_n$ is bounded in $H^1(\mathbb{R}^N)$.
This implies that there exist a subsequence of $(v_n)_n$,
which is still denoted by $(v_n)_n$,
and $v_0\in H^1(\mathbb{R}^N)$ such that $v_n\rightharpoonup v_0$ weakly in $H^1(\mathbb{R}^N)$.

Next, we show $v_0\ne0$.
Since $v_n\ne0$,
letting
\[
\lambda_n
=\left(\frac{\|\nabla v_n\|_{L^2}^2+\omega\|v_n\|_{L^2}^2}{\|v_n\|_{L^{p+1}}^{p+1}}\right)^{1/(p-1)}
=\left(\frac{L_\omega(v_n)+G(v_n)}{\|v_n\|_{L^{p+1}}^{p+1}}\right)^{1/(p-1)},
\]
then we have $\lambda_n>0$ and $K_\omega^0(\lambda_nv_n)=0$.
Moreover, by \eqref{eq:lnd}
and the weak continuity of the potential energy (cf.\ \cite[Theorem~11.4]{LL}),
we obtain
\begin{equation}\label{eq:lil}
\lim_{n\to\infty}\lambda_n
=\left(\frac{d(\omega)+\frac{p-1}{2(p+1)}G(v_0)}{d(\omega)}\right)^{1/(p-1)}.
\end{equation}
By Lemma~\ref{lem:do},
$K_\omega^0(\lambda_nv_n)=0$,
and the definition of $d^0(\omega)$,
it follows  that
\[
d(\omega)<d^0(\omega)
\le\frac{p-1}{2(p+1)}\|\lambda_nv_n\|_{L^{p+1}}^{p+1}
=\lambda_n^{p+1}\frac{p-1}{2(p+1)}\|v_n\|_{L^{p+1}}^{p+1}
\]
for all $n\in\mathbb{N}$.
Therefore, taking the limit,
by \eqref{eq:lnd}, 
\eqref{eq:lil},
and $d(\omega)>0$, we obtain $G(v_0)>0$.
This implies $v_0\ne0$.

Finally, we show the strong convergence of $(v_n)_n$ in $H^1(\mathbb{R}^N)$.
Taking a subsequence of $(v_n)_n$ if necessary,
we may assume that $v_n\to v_0$ a.e.\ in $\mathbb{R}^N$.
Then by using the Brezis--Lieb Lemma~\cite{BL83},
we have
\begin{align}\label{eq:bll}
L_\omega(v_n)
-L_\omega(v_n-v_0)
&\to L_\omega(v_0), \\ \label{eq:blk}
-K_\omega(v_n-v_0)
&\to K_\omega(v_0),
\end{align}
where we used $K_\omega(v_n)=0$ in \eqref{eq:blk}.
Since $L_\omega(v_0)>0$ by $v_0\ne0$,
it follows from \eqref{eq:bll}
and \eqref{eq:lnd} that
\[
\frac{p-1}{2(p+1)}\lim_{n\to\infty}L_\omega(v_n-v_0)
<\frac{p-1}{2(p+1)}\lim_{n\to\infty}L_\omega(v_n)
=d(\omega).
\]
From this and \eqref{eq:dkl},
we have $K_\omega(v_n-v_0)>0$ for large $n$.
Therefore, by \eqref{eq:blk}, we obtain $K_\omega(v_0)\le0$,
and thus, by \eqref{eq:dkl} and the weak lower semicontinuity of norms,
\[
d(\omega)
\le\frac{p-1}{2(p+1)}L_\omega(v_0)
\le\frac{p-1}{2(p+1)}\lim_{n\to\infty}L_\omega(v_n)
=d(\omega).
\]
This and \eqref{eq:bll} imply that $L_\omega(v_n-v_0)\to0$,
and therefore, $v_n\to v_0$ in $H^1(\mathbb{R}^N)$.
This completes the proof.
\end{proof}

Finally,
we give a useful lemma for the proof of Theorem~\ref{blowup}.

\begin{lemma} \label{vari-ineq}
Let $\phi\in\mathcal{G}_\omega$.
If $v\in H^1(\mathbb{R}^N)$ satisfies
$\|v\|_{L^{p+1}}=\|\phi\|_{L^{p+1}}$,
then the following hold.
\begin{enumerate}[label=\textup{(\roman{*})}, ref=\textup{\roman{*}}]
\item \label{vari-ineq1}
$K_\omega(v)\ge0$,
\item\label{vari-ineq2}
$S_\omega(v)\ge S_\omega(\phi)$.
\end{enumerate}
\end{lemma}

\begin{proof}
Inequality~\eqref{vari-ineq1} follows from Lemma~\ref{lem:k1} and 
$
d(\omega)=\frac{p-1}{2(p+1)}\|\phi\|_{L^{p+1}}^{p+1}
$.
Inequality~\eqref{vari-ineq2} follows from \eqref{skn} and \eqref{vari-ineq1}.
\end{proof}

\section{Blowup solutions}\label{sec:blowup}

In this section,
we prove Theorem~\ref{blowup}.
Throughout this section,
we impose the same assumption as in Theorem~\ref{blowup},
that is,
we assume \eqref{asmp}, $\omega>\omega_0$, and
\begin{equation}\label{p2el}
\partial_\lambda^2S_\omega(\phi_\omega^\lambda)|_{\lambda=1}
=\|\nabla\phi_\omega\|_{L^2}^2
-\frac{\alpha(\alpha-1)}{2}G(\phi_\omega)
-\frac{\beta(\beta-1)}{p+1}\|\phi_\omega\|_{L^{p+1}}^{p+1}
\le0,
\end{equation}
where $v^\lambda(x)=\lambda^{N/2}v(\lambda x)$, 
$G$ is defined in \eqref{eq:Gv}, 
and
\[
\beta
=\frac{N(p-1)}{2}.
\]
By using this notation,
we have
\begin{align}\label{S^lam}
S_\omega(v^\lambda)
&=\frac{\lambda^2}{2}\|\nabla v\|_{L^2}^2
+\frac\omega2\|v\|_{L^2}^2
-\frac{\lambda^\alpha}{2}G(v)
-\frac{\lambda^\beta}{p+1}\|v\|_{L^{p+1}}^{p+1}, \\
Q(v^\lambda) \label{Q^lam}
&=\lambda^2\|\nabla v\|_{L^2}^2
-\frac{\alpha\lambda^\alpha}{2}G(v)
-\frac{\beta\lambda^\beta}{p+1}\|v\|_{L^{p+1}}^{p+1}
=\lambda\partial_\lambda S_\omega(v^\lambda), \\
K_\omega(v^\lambda) \label{K^lam}
&=\lambda^2\|\nabla v\|_{L^2}^2
+\omega\|v\|_{L^2}^2
-\lambda^\alpha G(v)
-\lambda^\beta\|v\|_{L^{p+1}}^{p+1}.
\end{align}
Here,
we define
\[\mathcal{A}_\omega
=\{\,v\in H^1(\mathbb{R}^N)\mid 
S_\omega(v)<S_\omega(\phi_\omega),~
\|v\|_{L^2}\le\|\phi_\omega\|_{L^2},~
\|v\|_{L^{p+1}}>\|\phi_\omega\|_{L^{p+1}}\,\}.
\]
Recall that
\[
\mathcal{B}_\omega
=\{\,v\in\mathcal{A}_\omega\mid Q(v)<0\,\}.
\]

\begin{lemma}\label{inva-flow}
If $u_0\in\mathcal{A}_\omega$, then the solution $u(t)$ of \eqref{nls} with $u(0)=u_0$ satisfies $u(t)\in\mathcal{A}_\omega$ for all $t\in I_{\max}$.
\end{lemma}

\begin{proof}
Since $E$ and $\|\cdot\|_{L^2}$ are conserved quantities of \eqref{nls},
we have
$\|u(t)\|_{L^2}\le\|\phi_\omega\|_{L^2}$ and
$S_\omega(u(t))<S_\omega(\phi_\omega)$
for all $t\in I_{\max}$.
By Lemma~\ref{vari-ineq}~\eqref{vari-ineq2}, it follows that $\|u(t)\|_{L^{p+1}}\ne\|\phi_\omega\|_{L^{p+1}}$ for all $t\in I_{\max}$.
Therefore, 
by $\|u_0\|_{L^{p+1}}>\|\phi_\omega\|_{L^{p+1}}$ and the continuity of the solution $u(t)$,
we obtain $\|u(t)\|_{L^{p+1}}>\|\phi_\omega\|_{L^{p+1}}$
for all $t\in I_{\max}$.
This completes the proof.
\end{proof}

The following is the key lemma for our proof.

\begin{lemma}\label{key-lemma}
Let $v\in H^1(\mathbb{R}^N)$ satisfy
\[
\|v\|_{L^2}\le\|\phi_\omega\|_{L^2},\quad
\|v\|_{L^{p+1}}\ge\|\phi_\omega\|_{L^{p+1}},\quad
Q(v)\le0.
\]
Then
\begin{equation}\label{key-ineq}
\frac{Q(v)}{2}
\le S_\omega(v)-S_\omega(\phi_\omega).
\end{equation}
In particular,
if $u_0\in\mathcal{B}_\omega$, then the solution $u(t)$ of \eqref{nls} with $u(0)=u_0$ satisfies $u(t)\in\mathcal{B}_\omega$ for all $t\in I_{\max}$.
\end{lemma}

\begin{proof}
Let
\[
\lambda_0
=\left(\frac{\|\phi_\omega\|_{L^{p+1}}^{p+1}}{\|v\|_{L^{p+1}}^{p+1}}\right)^{1/\beta}.
\]
Then we have
\[
0<\lambda_0\le1,\quad
\|v^{\lambda_0}\|_{L^2}
=\|v\|_{L^2}
\le\|\phi_\omega\|_{L^2},\quad
\|v^{\lambda_0}\|_{L^{p+1}}^{p+1}
=\lambda_0^\beta\|v\|_{L^{p+1}}^{p+1}
=\|\phi_\omega\|_{L^{p+1}}^{p+1}.
\]
Here, we define 
\begin{align*}
f(\lambda)
&=S_\omega(v^\lambda)
-\frac{\lambda^2}{2}Q(v) \\
&=-\frac12\left(\lambda^\alpha-\frac{\alpha\lambda^2}{2}\right)G(v)
+\frac{\omega}{2}\|v\|_{L^2}^2
-\frac{1}{p+1}\left(\lambda^{\beta}-\frac{\beta\lambda^2}{2}\right)\|v\|_{L^{p+1}}^{p+1}
\end{align*}
for $\lambda\in(0,1]$.
If we have $f(\lambda_0)\le f(1)$,
then it follows from Lemma~\ref{vari-ineq}~\eqref{vari-ineq2},
$Q(v)\le0$,
and $f(\lambda_0)\le f(1)$ that
\begin{equation} \label{conc}
S_\omega(\phi_\omega)
\le S_\omega(v^{\lambda_0})
\le S_\omega(v^{\lambda_0})-\frac{\lambda_0^2}{2}Q(v)
\le S_\omega(v)-\frac{Q(v)}{2},
\end{equation}
which is the desired inequality~\eqref{key-ineq}.

In what follows, we prove $f(\lambda_0)\le f(1)$,
which is rewritten as
\begin{align}\label{aim}
G(v)
\le\frac{2(2\lambda_0^\beta-\beta\lambda_0^2-2+\beta)}{(p+1)(\alpha\lambda_0^2-2\lambda_0^\alpha-\alpha+2)}
\|v\|_{L^{p+1}}^{p+1}.
\end{align}
By $\alpha K_\omega(\phi_\omega)
-(\alpha+1)Q(\phi_\omega)=0$
and \eqref{p2el},
we have
\begin{align*}
\alpha\omega\|\phi_\omega\|_{L^2}^2
&=\|\nabla\phi_\omega\|_{L^2}^2
-\frac{\alpha(\alpha-1)}{2}G(\phi_\omega)
+\left(\alpha-\frac{\beta(\alpha+1)}{p+1}\right)\|\phi_\omega\|_{L^{p+1}}^{p+1} \\
&\le\left(\alpha+\frac{\beta(\beta-\alpha-2)}{p+1}\right)\|\phi_\omega\|_{L^{p+1}}^{p+1}.
\end{align*}
Therefore, it follows from $\|v\|_{L^2}\le\|\phi_\omega\|_{L^2}$
and $\|\phi_\omega\|_{L^{p+1}}^{p+1}
=\lambda_0^\beta\|v\|_{L^{p+1}}^{p+1}$ that
\begin{equation}\label{ineq1}
\omega\|v\|_{L^2}^2
\le\left(1+\frac{\beta(\beta-\alpha-2)}{(p+1)\alpha}\right)\lambda_0^\beta\|v\|_{L^{p+1}}^{p+1}.
\end{equation}
By using Lemma~\ref{vari-ineq}~\eqref{vari-ineq1} for $v^{\lambda_0}$,
\eqref{K^lam},
\eqref{ineq1},
and $Q(v)\le0$,
we have
\begin{align*}
G(v)
&\le\lambda_0^{2-\alpha}\|\nabla v\|_{L^2}^2
+\lambda_0^{-\alpha}\omega\|v\|_{L^2}^2
-\lambda_0^{\beta-\alpha}\|v\|_{L^{p+1}}^{p+1} \\ \label{ineq2}
&\le\lambda_0^{2-\alpha}\|\nabla v\|_{L^2}^2
+\frac{\beta(\beta-\alpha-2)}{(p+1)\alpha}\lambda_0^{\beta-\alpha}\|v\|_{L^{p+1}}^{p+1} \\
&\le\frac{\alpha}{2}\lambda_0^{2-\alpha}G(v)
+\frac{\beta}{p+1}\left(\lambda_0^{2-\alpha}
+\frac{\beta-\alpha-2}{\alpha}\lambda^{\beta-\alpha}\right)\|v\|_{L^{p+1}}^{p+1},
\end{align*}
and thus,
\begin{equation}\label{ineq3}
G(v)
\le\frac{2\beta}{(p+1)(2-\alpha\lambda_0^{2-\alpha})}\left(\lambda_0^{2-\alpha}
+\frac{\beta-\alpha-2}{\alpha}\lambda^{\beta-\alpha}\right)\|v\|_{L^{p+1}}^{p+1}.
\end{equation}
In view of \eqref{aim} and \eqref{ineq3},
we only have to show that
\begin{equation*} 
\frac{\beta}{2-\alpha\lambda^{2-\alpha}}\left(\lambda_0^{2-\alpha}
+\frac{\beta-\alpha-2}{\alpha}\lambda^{\beta-\alpha}\right)
\le\frac{2\lambda_0^\beta-\beta\lambda_0^2-2+\beta}{\alpha\lambda_0^2-2\lambda_0^\alpha-\alpha+2}
\end{equation*}
for all $\lambda\in(0,1)$,
which is equivalent to
\[
g_1(\lambda)
\mathrel{\mathop:}=\frac{(2-\alpha\lambda^{2-\alpha})(2\lambda^\beta-\beta\lambda^2-2+\beta)}{\beta\lambda^{\beta-\alpha}(\alpha\lambda^2-2\lambda^\alpha-\alpha+2)}
-\frac{1}{\lambda^{\beta-2}}
-\frac{\beta-\alpha-2}{\alpha}
\ge0.
\]
Since $\lim_{\lambda\nearrow1}g_1(\lambda)=0$, it suffices to show that 
\begin{align*}
g_1'(\lambda)
=\frac{2(1-\lambda^{2-\alpha})}{\beta\lambda^{\beta-\alpha+1}(\alpha\lambda^2-2\lambda^\alpha-\alpha+2)^2}
\begin{aligned}[t]
&\big(2\alpha(2-\alpha)\lambda^\beta
-\alpha\beta(\beta-\alpha)\lambda^2 \\
&\quad+2\beta(\beta-2)\lambda^\alpha
-(\beta-\alpha)(\beta-2)(2-\alpha)\big)
\le0
\end{aligned}
\end{align*}
for all $\lambda\in(0,1)$,
which holds if we have
\begin{equation*}\label{aim3}
g_2(\lambda)
\mathrel{\mathop:}=2\alpha(2-\alpha)\lambda^\beta
-\alpha\beta(\beta-\alpha)\lambda^2
+2\beta(\beta-2)\lambda^\alpha
-(\beta-\alpha)(\beta-2)(2-\alpha)
\le0.
\end{equation*}
Since $g_2(1)=0$,
it is enough to show that
\[
g_2'(\lambda)
=2\alpha\beta\lambda^{\alpha-1}\left(
(2-\alpha)\lambda^{\beta-\alpha}
-(\beta-\alpha)\lambda^{2-\alpha}
+\beta-2)\right)
\ge0
\]
for all $\lambda\in(0,1)$.
This is equivalent to
\[
g_3(\lambda)
\mathrel{\mathop:}=(2-\alpha)\lambda^{\beta-\alpha}
-(\beta-\alpha)\lambda^{2-\alpha}
+\beta-2\ge0.
\]
Since $g_3(1)=0$ and
\[
g_3'(\lambda)
=-(\beta-\alpha)(2-\alpha)\lambda^{1-\alpha}(1-\lambda^{\beta-2})\le0,
\]
we have $g_3(\lambda)\ge0$ for all $\lambda\in(0,1)$.
Therefore,
we obtain $f(\lambda_0)\le f(1)$,
and thus, the inequality \eqref{conc} follows.

The last claim of Lemma~\ref{key-lemma} follows from Lemma~\ref{inva-flow} and \eqref{key-ineq}.
This completes the proof.
\end{proof}

\begin{proof}[Proof of Theorem~\ref{blowup}]
Let $u_0\in\mathcal{B}_\omega\cap\Sigma$
and $u(t)$ be the solution of \eqref{nls} with $u(0)=u_0$.
Then by the virial identity~\eqref{virial},
Lemma~\ref{key-lemma},
and the conservation of $S_\omega$,
we have
\begin{equation}\label{viri-ineq}
\frac{d^2}{dt^2}\|xu(t)\|_{L^2}^2
=8Q(u(t))
\le16\big(S_\omega(u(t))-S_\omega(\phi_\omega)\big)
=16\big(S_\omega(u_0)-S_\omega(\phi_\omega)\big)
<0
\end{equation}
for all $t\in I_{\max}$.

If $T^+=\infty$,
then it follows from \eqref{viri-ineq} that $\|xu(t)\|_{L^2}$ becomes negative for large $t$.
This is a contradiction.
Thus, the solution $u(t)$ blows up in finite time.
\end{proof}

\section{Strong instability of standing waves}\label{sec:SI}

In this section, we prove Theorem~\ref{mainthm}.
Here, we impose the same assumption as in Theorem~\ref{mainthm}.

\begin{lemma}\label{pib}
$\phi_\omega^\lambda\in\mathcal{B}_\omega$ for all $\lambda>1$.
\end{lemma}

\begin{proof}
By the definition of the scaling 
$\lambda\mapsto\phi_\omega^\lambda$,
we have
$\|\phi_\omega^\lambda\|_{L^2}=\|\phi_\omega\|_{L^2}$
and $\|\phi_\omega^\lambda\|_{L^{p+1}}=\lambda^{\beta/(p+1)}\|\phi_\omega\|_{L^{p+1}}>\|\phi_\omega\|_{L^{p+1}}$
for all $\lambda>1$,
where $\beta=N(p-1)/2>2$.

Now, we show that $S_\omega(\phi_\omega^\lambda)<S_\omega(\phi_\omega)$ and
$Q(\phi_\omega^\lambda)<0$ for all $\lambda>1$.
In view of \eqref{S^lam},
the function $S_\omega(\phi_\omega^\lambda)$ of $\lambda$ has the form 
$S_\omega(\phi_\omega^\lambda)
=A\lambda^2+B-C\lambda^\alpha-D\lambda^\beta$
with some $A,B,C,D>0$.
By $\partial_{\lambda}S_\omega(\phi_\omega^\lambda)|_{\lambda=1}=0$ and the assumption $\partial_{\lambda}^2S_\omega(\phi_\omega^\lambda)|_{\lambda=1}\le0$, we have
$-\beta(\beta-2)D
\le-\alpha(2-\alpha)C$.
This leads to
\begin{align*}
\partial_{\lambda}^3S_\omega(\phi_\omega^\lambda)
&=\alpha(\alpha-1)(2-\alpha)C\lambda^{\alpha-3}
-\beta(\beta-1)(\beta-2)D\lambda^{\beta-3} \\
&\le-\alpha(2-\alpha)\lambda^{\alpha-3}\left((\beta-1)\lambda^{\beta-\alpha}-(\alpha-1)\right)C
<0
\end{align*}
for all $\lambda>1$.
Therefore,
for $\lambda>1$,
it follows that 
$\partial_\lambda^2S_\omega(\phi_\omega^\lambda)<\partial_\lambda^2S_\omega(\phi_\omega^\lambda)|_{\lambda=1}\le0$,
$\partial_\lambda S_\omega(\phi_\omega^\lambda)<
\partial_\lambda S_\omega(\phi_\omega^\lambda)|_{\lambda=1}=0$,
and thus
$S_\omega(\phi_\omega^\lambda)<S_\omega(\phi_\omega)$.
Moreover, by differentiating \eqref{Q^lam},
we have
$\partial_{\lambda}Q(\phi_\omega^\lambda)
=\partial_{\lambda}S_\omega(\phi_\omega^\lambda)
+\lambda\partial_{\lambda}^2S_\omega(\phi_\omega^\lambda)
<0$.
This implies $Q(\phi_\omega^\lambda)<Q(\phi_\omega)=0$.
This completes the proof.
\end{proof}

Now, we prove the main theorem.

\begin{proof}[Proof of Theorem~\ref{mainthm}]

Let $\varepsilon>0$.
Then since $\phi_\omega^\lambda\to\phi_\omega$ in $H^1(\mathbb{R}^N)$ as $\lambda\searrow1$,
there exists $\lambda_0>1$ such that $\|\phi_\omega-\phi_\omega^{\lambda_0}\|_{H^1}<\varepsilon/2$.
Let $\chi\in C^\infty[0,\infty)$ be a function satisfying 
$0\le\chi\le1$,
$\chi(r)=1$ if $0\le r\le1$, and
$\chi(r)=0$ if $r\ge2$.
For $M>0$, 
we define a cutoff function $\chi_M\in C_\mathrm{c}^\infty(\mathbb{R}^N)$ by 
$\chi_M(x)=\chi(|x|/M)$.
Then we see that 
$\chi_M\phi_\omega^{\lambda_0}\to\phi_\omega^{\lambda_0}$ in $H^1(\mathbb{R}^N)$ as $M\to\infty$.
Moreover, we have
$\chi_M\phi_\omega^{\lambda_0}\in\Sigma$ and
$\|\chi_M\phi_\omega^{\lambda_0}\|_{L^2}\le\|\phi_\omega^{\lambda_0}\|_{L^2}=\|\phi_\omega\|_{L^2}$ for all $M>0$.
Therefore, by Lemma~\ref{pib} and the continuity of $S_\omega$, $\|\cdot\|_{L^{p+1}}$, and $Q$,
there exists $M_{0}>0$ such that 
$\|\phi_\omega^{\lambda_0}-\chi_{M_0}\phi_\omega^{\lambda_0}\|_{H^1}<\varepsilon/2$ and
$\chi_{M_0}\phi_\omega^{\lambda_0}\in\mathcal{B}_\omega\cap\Sigma$.
Thus, 
we obtain 
$\|\chi_{M_0}\phi_\omega^{\lambda_0}-\phi_\omega\|_{H^1}
<\varepsilon$,
and by Theorem~\ref{blowup}, the solution $u(t)$ with $u(0)=\chi_{M_0}\phi_\omega^{\lambda_0}$ blows up in finite time.
Hence,
the standing wave solution $e^{i\omega t}\phi_\omega$ of \eqref{nls} is strongly unstable.
\end{proof}

\paragraph*{Acknowledgements}

The authors would like to express their deepest gratitude to Yusuke Shimabukuro for valuable discussions.
This work was supported by JSPS KAKENHI Grant Numbers 15K04968 and 26247013.


Noriyoshi Fukaya

Department of Mathematics, Graduate School of Science, Tokyo University of Science, 
1-3 Kagurazaka, Shinjuku-ku, Tokyo 162-8601, Japan 

\textit{E-mail address}: \texttt{1116702@ed.tus.ac.jp}

\vspace\baselineskip

Masahito Ohta

Department of Mathematics, Tokyo University of Science, 
1-3 Kagurazaka, Shinjuku-ku, Tokyo 162-8601, Japan 

\textit{E-mail address}: \texttt{mohta@rs.tus.ac.jp}

\end{document}